\documentclass[12pt]{amsart}
\usepackage{amsmath}
\usepackage{amssymb}
\usepackage{longtable}
\usepackage{siunitx}
\usepackage{booktabs} 
\usepackage{tabularx} 
\usepackage{tabu} 
\usepackage{graphics}
\usepackage{cite}
\newtheorem{theorem}{Theorem}
\theoremstyle{plain}

\newtheorem{lemma}{Lemma}

\numberwithin{equation}{section}
\numberwithin{lemma}{section}
\numberwithin{theorem}{section}
\numberwithin{corollory}{section}
\numberwithin{proposition}{section}
\begin{document}

	\title{Unit  Groups of   Group Algebras of Abelian Groups of order 32}
	\author{Suchi Bhatt}
	\address{Department of Mathematics and Scientific CompuTing, M. M. M. University of Technology, Gorakhpur,U. P.,  India}
	\email{1995suchibhatt@gmail.com}
	\author{Harish Chandra}
	\address{Department of Mathematics and Scientific CompuTing, M. M. M. University of Technology, Gorakhpur, U. P.,  India}
	\email{hcmsc@mmmut.ac.in}		
	\maketitle
	\begin{abstract}
	
	Let $F$  be a finite field of characteristic $p>0$  with $q = p^{n}$ elements. In this paper, a complete characterization of the unit groups $U(FG)$ of group algebras $FG$ for the  abelian groups of order $32$, over finite field of characteristic $p>0$ has been  obtained.
\end{abstract}
  \text{Key words:}\; Group algebras, Unit groups,  Jacobson radical.\\
   \text{Mathematics Subject Classification \text{(2010)}}: 16S34, 17B30.

     \section{Introduction}
     Let $FG$ be the group algebra of a group $G$ over a field $F$. Suppose $U(FG)$ be the group of all invertible elements of the group algebra $FG$, called unit group of $FG$. In this paper, we   study   the unit groups of group algebra for abelian groups of order $32$. Suppose $V(FG)$ be the normalized unit group, $\omega(G)$ be the augmentation ideal of $G$,  $J(FG)$ is the Jacobson radical of the group algebra and $V = 1+ J(FG)$. It is known   fact that $U(FG) \cong V(FG) \times F^{*}$.
     An element $g\in G$ is called $p$-regular if $(p,o(g)) = 1$, where Char$F$ = $p>0$.
     Now let   $m$ be the L.C.M. of the orders of  all the $p$-regular elements of G and $\xi$ be a primitive $m$-th root of unity. Suppose $T$ is  the multiplicative group consisting of those integers $t$, taken modulo $m$, which gives    $\xi  \rightarrow \xi^{t}$   an  F-automorphism of $F(\xi)$ over $F$. Let $g_1, g_2 \in G$ are two $p$-regular elements. These are  said to be $F$-conjugate if $g_1^{t} = x^{-1}g_2x$, where $x\in G$ and $t\in T$. 
     This defines  an equivalence relation, so we have a partitions of  the $p$-regular elements of $G$  into $p$-regular, $F$-conjugacy classes. Our problem is based on the   Witt-Berman theorem  \cite[Ch.17, Theorem 5.3]{Gkarp2}, which states that the  number of non-isomorphic simple $FG$-modules is equal to the number of $F$-conjugacy classes of $p$-regular elements of $G$.
     Problem  of finding unit groups of group  algebras   generated a considerable interest in recent decade and can be easily seen in \cite{Bhatt,gkarpi1,Mkhan1,Khan2007,NehashaJb2,Sharma2006,Gtang1,Tang2014}. Recently in \cite{Sahai2020,saha2020}, Sahai and Ansari have characterized the unit groups of group algebras for the  abelian  groups of orders  up to 20.
     Let $G$ be a group of order $32$,  we have seven  non-isomorphic abelian groups $C_{32}$, $C_{16}\times C_{2}$, $C_{8}\times C_{4}$, $C_{8}\times C_{2}\times C_{2}$, $C_{4}\times C_{4}\times C_{2}$, $C_{4}\times  C^{3}_{2} $ and  $C^{5}_{2}$. We have completely obtained  the structure  of the unit groups  of the group algebras for all these seven  groups over any finite field of characteristics $p>0$.   
     Here, we denote  $M(n, F)$  the algebra of all $n\times n$ matrices over $F$,
     $GL(n,F)$ is the general linear group of degree $n$ over $F$, $CharF$   the characteristic of $F$,  $C_{n}$ is the cyclic group of order $n$ and $F^{*}= F  \setminus \{0\}$.
     \section{Preliminaries}
     We use the following results frequently throughout  our work. 
     \begin{lemma}  \cite[Proposition 1.2]{rafarz1} \label{l12}
     	The number of simple components of
     	$FG / J(FG)$ is equal to the number of cyclotomic $F$-classes in $G$.
     \end{lemma}
     \begin{lemma}  \cite[Lemma \ref{l12}]{Creedon2008}  \label{l1}
     	Let $F$ be a finite field of characteristic $p$ with
     	$|F| = q = p^{n}$. Then $U(FC^k_{p}) = C_{p}^{n(p^k -1)} \times C_{p^n-1}$.
     	
     \end{lemma}
     \begin{lemma}   \cite[Lemma 2.3]{NehaShajb1}    \label{l2}
     	Let F be a finite field of characteristic $p $ with
     	$|F| = q = p^n$. Then 
     	
     	\[
     	U(FC_{p^k}) \cong \left.
     	\begin{cases}
     	
     	C_{p}^{n(p -1)} \times C_{p^n-1}    & \text{if} \, \, k = \,1; \\
     	
     	\prod_{s=1}^{k}C^{h_s}_{p^s}\times C_{p^n-1}  , & \text{otherwise},\\

     	\end{cases}  \right.   \]
     	
     	where $h_{k} = n(p-1)$ and $h_{s} = n p^{k-s-1} (p-1)^2 $  for all $s$,  $1\leq s < k.$
     \end{lemma}

     \begin{lemma}  \cite{CPMSK} \label{l4} 
     	Let $G$ be a group and  $R$ be a commutative ring. Then the set of all finite class sums forms an $R$-basis of $ \zeta(RG)$, the center
     	of $RG$.
     	
     \end{lemma}
     
     \begin{lemma}  \cite{CPMSK} \label{l5}
     	Let $FG$ be a semi-simple group algebra. If $G'$ denotes the
     	commutator subgroup of $G$, then
     	$$  FG = FG_{e_{G'}} \oplus \Delta(G,G')$$
     	where $FG_{e_{G'}} \cong F(G/G')$ is the sum of all commutative simple components of $FG$ and $\Delta(G,G')$ is the sum of all the others.

     \end{lemma}

     \section{Main Results}
     
     \begin{theorem} \label{th1}
     	Let $F$ be a finite field of characteristic $p>0$, having  $ q = p^{n}$ elements and $G\cong C_{32}$.   
     	\begin{enumerate}

     		\item If $p=2$. Then,

     		$$  U(FC_{32}) \cong C^{n}_{32}  \times C^{n}_{16}   \times C^{2n}_{8} \times C^{4n}_{4} \times C^{8n}_{2} \times C_{2^n -1}.$$

     		\item If $p\neq 2 $. Then,
     	\end{enumerate}

     	\[
     	U(FC_{32}) \cong \left.
     	\begin{cases}
     	
     	C^{32}_{p^n-1},   & \text{if} \, \, q \, \equiv \, 1\, mod \,\, 32;\\
     	
     	C^{2}_{p^n-1}\times C^{15}_{p^{2n}-1}, & \text{if} \, \, q \, \equiv \, -1 \, mod \,\, 32;\\
     	
     	C^{2}_{p^{8n}-1}  \times C^{2}_{p^{4n}-1} \times C^{3}_{p^{2n}-1} \times C^{2}_{p^n-1}, & \text{if} \, \, q\equiv \, 3,\,  -5, \,  11,\,  -13\, \, mod \, 32;\\
     	
     	C^{2}_{p^{8n}-1}  \times C^{2}_{p^{4n}-1} \times C^{2}_{p^{2n}-1} \times C^{4}_{p^n-1}, & \text{if} \, \, q\equiv \, -3,\,   5, \,  -11,\,  13\, \, mod \, 32;\\
     	
     	C^{2}_{p^n-1}\times C^{4}_{p^{4n}-1} \times  C^{7}_{p^{2n}-1},
     	& \text{if} \, \, q \, \equiv \, 7 \, mod \,\, 32;\\
     	
     	C^{8}_{p^n-1}\times C^{4}_{p^{2n}-1} \times  C^{4}_{p^{4n}-1},
     	& \text{if} \, \, q \, \equiv \, -7 \, mod \,\, 32;\\
     	
     	C^{2}_{p^n-1} \times  C^{15}_{p^{2n}-1},
     	& \text{if} \, \, q \, \equiv \, 15 \, mod \,\, 32;\\
     	
     	C^{16}_{p^n-1} \times  C^{8}_{p^{2n}-1},
     	& \text{if} \, \, q \, \equiv \, -15\, mod \,\, 32.\\
     	
     	\end{cases}  \right.   \]

     \end{theorem}
     \begin{proof}
     	The presentation of  $C_{32}$ is given by 
     	$$C_{32} = < a\, \, |\, \, a^{32} = 1>.$$
     	\begin{enumerate}
     		\item If $p=2$, then $|F| = q = 2^{n}$. Since $G \cong  C_{32} \cong C_{2^{5}}$, therefore using Lemma \ref{l2}, we have 
     		$$  U(FC_{32}) \cong C^{n}_{32}  \times C^{n}_{16}   \times C^{2n}_{8} \times C^{4n}_{4} \times C^{8n}_{2} \times C_{2^n -1}.$$

     		\item If  $p \neq 2 $, then  $p$ does not divides $|C_{32}|$, therefore by Maschke's theorem, $FC_{32}$ is semisimple over $F$. Hence by    Wedderburn decomposition theorem  and by Lemma \ref{l5}, we have 
     		$$ FC_{32} \cong  (\bigoplus^{r}_{i=1}  M(n_{i},D_{i}))$$	where for each i, $n_{i}\geq1$  and 	$ D{i}$'s are finite field extensions of $F$. 
     		Since  group is abelian, therefore dimension constraint gives 	$n_{i}=1$, for every $ i.$ It is clear that $C_{32}$ has $32$ conjugacy classes.
     		Now  for any $k\in N$,  $x^{q^{k}}$ = $x$, $\forall x\in \zeta(FC_{32})$ if and only  $\widehat{{C}^{q^t}_{i}} = \widehat{{C}_{i}}$, for all $1\leq i\leq 32$. It   exists if and only if $32|q^k-1$ or $32|q^k+1$. If $D_{i}^* = <y_{i}>$ for all $i$, $1\leq i\leq r$, then $x^{q^{k}} = x$,  $\forall x\in \zeta(FC_{32})$ if and only if $y^{q^{k}}_{i} = 1$, which holds if and only if $[D_{i}:F]|k$, for all $1\leq i\leq r$. Hence the least number $t$ such that $32|q^k-1$ or $32|q^k+1$, 
     		$$t = l.c.m.\{[D_{i}:F] |  1\leq i \leq r \}.$$

     		Therefore  all conjugacy classes of $C_{32}$ are $p$-regular and $m$=32. By    observation we have following possibilities for $q$:
     	\end{enumerate}
     	(a)	If $q\equiv \,1\,  mod\,  32$,    then $t = 1$; 	\\
     	(b) If $q\equiv \,-1\,  mod \, 32$,   then $t = 2 $;  \\
     	(c) If $q\equiv \, 3,\,  -5, \,  11,\,  -13\, \, mod \, 32$;,  then $t = 8;$ \\
     	(d) If $q\equiv \, -3,\,   5, \,  -11,\,   13\, \, mod \, 32$;,  then $t = 8; $ \\
     	(e) If $q \equiv \,  7  \,\,  mod \, \, 32$,  then $t = 4 $; \\
     	(f) If $q \equiv  \, -7  \,\,  mod \, \, 32$,  then $t = 4 $; \\
     	(g) If $q \equiv  \, 15  \,\,  mod \, \, 32$,  then $t = 2 $; \\
     	(h) If $q \equiv  \, -15  \,\,  mod \, \, 32$,  then $t = 2 $.\\
     	
     	Now we will find $T$ and  the number of  $p$-regular $F$-conjugacy classes,  denoted  by $c$. By   Lemma \ref{l4},  $dim_{F}(\zeta(FC_{32})) = 32$, therefore  $\sum_{i=1}^{r} [D_{i} : F] = 32$. We  have   the  following cases:
     	
     	\begin{enumerate}
     		\item If $q\equiv \,1\,  mod\,  32$,    then $T = \{1\}\,mod\,\,32.$ Thus $p$-regular $F$-conjugacy classes are the conjugacy classes of $C_{32} $ and $c$=32.   Hence $FC_{32}\cong  F^{32}.$  
     		
     		\item If  $q\equiv \,-1\,  mod \, 32$,   then $T = \{1,-1\}\,mod\,\,32.$ Thus $p$-regular $F$-conjugacy classes are $\{1\}$,  $\{a^{16}\}$,  $\{a^{\pm i}\}$, $1\leq i \leq 15 $ and  $c$=17. Hence $FC_{32}\cong \,F^{2} \oplus F^{15}_{2}$.
     		
     		\item If $q\equiv \, 3,\,  -5, \,  11,\,  -13\, \, mod \, 32$,  then $T=  \{1, 3, 9, 11,  17, 19, 25, 27\}\,mod\,\,32.$ Thus $p$-regular $F$-conjugacy classes are $\{1\}$, $\{a,  a^{3}, a^{9}, a^{11}, a^{17}, a^{19}, a^{25}, a^{27}\}$,  $\{a^{2}, a^{6}, a^{18}, a^{22}\}$, $\{a^{4}, a^{12}\}$, $\{a^{5},  a^{7}, a^{13}, a^{15}, a^{21}, a^{23}, a^{29}, a^{31}\}$,  $\{a^{8},  a^{24}\}$,\\ $\{a^{10}, a^{14}, a^{26}, a^{30}\}$, $\{a^{16}\}$,  $\{a^{20},  a^{28}\}$ and   $c$=9.	  Hence $FC_{32}\cong   F^{2}_{8}    \oplus  F^{2}_{4}   \oplus F^{3}_{2} \oplus  F^{2}.$
     		
     		\item If $q\equiv \, -3,\,   5, \,  -11,\,   13\, \, mod \, 32$,  then $T=  \{1, 5, 9, 13,  17, 21, 25, 29\}\,mod\,\,32.$ Thus $p$-regular $F$-conjugacy classes are $\{1\}$, $\{a,  a^{5}, a^{9}, a^{13}, a^{17}, a^{21}, a^{25}, a^{29}\}$,  $\{a^{2}, a^{10}, a^{18}, a^{26}\}$, $\{a^{4}, a^{20}\}$, $\{a^{3},  a^{7}, a^{11}, a^{15}, a^{19}, a^{23}, a^{27}, a^{31}\}$,  $\{a^{8}\}$, \\ $\{a^{6}, a^{14}, a^{22}, a^{30}\}$, $\{a^{16}\}$,  $\{a^{24}\}$,  $\{a^{12},  a^{28}\}$ and   $c$=10.	  Hence $FC_{32}\cong   F^{2}_{8}    \oplus  F^{2}_{4}   \oplus F^{2}_{2} \oplus  F^{4}.$
     		
     		\item If $q\equiv 7 \,  mod \, 32 $,    then $T = \{1, 7, 17, 23\}\,mod\,\,32.$ Thus, $p$-regular $F$-conjugacy classes are
     		\{1\},  $\{a, a^{7}, a^{17}, a^{23}\}$, $\{ a^{2}, a^{14}\}$, $\{a^{3}, a^{5}, a^{19}, a^{21}\}$,  $\{ a^{4}, a^{28}\}$, $\{ a^{6}, a^{10}\}$,  $\{ a^{8}, a^{24}\}$, $\{a^{9}, a^{15}, a^{25}, a^{31}\}$, $\{a^{11}, a^{13}, a^{27}, a^{29}\}$, $\{a^{12}, a^{20}\}$, $\{a^{16}\}$,  $\{a^{18}, a^{30}\}$,  $\{a^{22}, a^{26}\}$ and  $c= 13$.
     		Hence $FC_{32}\cong  F^{2} \oplus F^{4}_{4} \oplus F^{7}_{2} .$  
     		
     		\item If $q\equiv -7 \,  mod \, 32 $,    then $T = \{1, 9, 17, 25\}\,mod\,\,32.$ Thus, $p$-regular $F$-conjugacy classes are
     		\{1\},  $\{a, a^{9}, a^{17}, a^{25}\}$, $\{ a^{2}, a^{18}\}$, $\{a^{3}, a^{11}, a^{19}, a^{27}\}$,  $\{ a^{4}\}$,     $\{ a^{6}, a^{22}\}$, $\{a^{5}, a^{13}, a^{21}, a^{29}\}$, $\{a^{7}, a^{15}, a^{23}, a^{31}\}$,  $\{ a^{8}\}$,  $\{a^{10}, a^{26}\}$,  $\{a^{12}\}$, $\{a^{16}\}$,  $\{a^{14}, a^{30}\}$,   $\{ a^{20}\}$,  $\{ a^{24}\}$,  $\{a^{28}\}$ and  $c= 16$.
     		Hence $FC_{32}\cong  F^{8} \oplus F^{4}_{2} \oplus F^{4}_{4} .$  
     		
     		\item If $q\equiv 15\,mod\,32$,  then $T = \{1, 15\}\, mod\,\,32.$ Thus, $p$-regular $F$-conjugacy classes are 	\{1\},
     		$\{ a, a^{15}\}$, $\{ a^{2}, a^{30}\}$, $\{ a^{3}, a^{13}\}$, $\{ a^{4}, a^{28}\}$, $\{ a^{5}, a^{11}\}$, $\{ a^{6}, a^{26}\}$, $\{ a^{7}, a^{9}\}$, $\{ a^{8}, a^{24}\}$, $\{ a^{10}, a^{22}\}$, $\{ a^{12}, a^{20}\}$, $\{ a^{14}, a^{18}\}$, $\{ a^{17}, a^{31}\}$, $\{ a^{19}, a^{29}\}$, $\{ a^{21}, a^{27}\}$, $\{ a^{23}, a^{25}\}$, $\{a^{16}\}$ and  $c$=17.
     		Hence, $FC_{32}\cong   F^{2}\oplus F^{15}_{2}.$  
     		\item If $q\equiv -15\,mod\,32$,  then $T = \{1, 17\}\, mod\,\,32.$ Thus, $p$-regular $F$-conjugacy classes are 	\{1\},
     		$\{ a, a^{17}\}$, $\{ a^{2}\}$, $\{a^{30}\}$, $\{ a^{3}, a^{19}\}$, $\{ a^{4}\}$, $\{a^{28}\}$, $\{ a^{5}, a^{21}\}$, $\{ a^{6}\}$, $\{a^{26}\}$, $\{ a^{7}, a^{23}\}$, $\{ a^{8}\}$, $\{a^{24}\}$,  $\{ a^{9}, a^{25}\}$, $\{ a^{10}\}$, $\{a^{22}\}$, $\{ a^{11}, a^{27}\}$, $\{ a^{13}, a^{29}\}$, $\{ a^{15}, a^{31}\}$,  $\{ a^{12}\}$, $\{a^{20}\}$,  $\{a^{16}\}$, $\{a^{14}\}$, $\{a^{18}\}$ and  $c$=24.
     		Hence, $FC_{32}\cong   F^{16}\oplus F^{8}_{2}.$ 
     		Thus our result follows.
     	\end{enumerate}
     \end{proof}
     \begin{theorem}
     	Let $F$ be a finite field of characteristic $p>0$ having  $q= p^{n}$ elements and $G \cong C_{16} \times C_{2}.$
     	\begin{enumerate}
     		\item If $p=2$. Then,
     		$ U(F[C_{16} \times C_{2}])  \cong  C^{n}_{16}\times C^{n}_{8} \times C^{2n}_{4} \times C^{20n}_{2}  \times C_{2^n -1}. $
     		
     		\item If $p\neq 2$. Then, 
     		
     	\end{enumerate}

     	\[
     	U(F[C_{16} \times C_{2}]) \cong \left.
     	\begin{cases}
     	
     	C^{32}_{p^n-1},    & \text{if} \, \, q \, \equiv \, 1\, mod \,\, 16;\\
     	
     	C^{4}_{p^n-1}\times C^{14}_{p^{2n}-1} , & \text{if} \, \, q \, \equiv \, -1 \, mod \,\, 16;\\
     	
     	C^{4}_{p^n-1}\times C^{6}_{p^{2n}-1}\times C^{4}_{p^{4n}-1} , & \text{if} \, \, q \, \equiv \,3,  -5\, mod \,\, 16;\\
     	
     	C^{8}_{p^n-1}\times C^{4}_{p^{2n}-1}\times C^{4}_{p^{4n}-1} , & \text{if} \, \, q \, \equiv \,-3,   5\, mod \,\, 16;\\
     	
     	C^{4}_{p^{n}-1}\times C^{14}_{p^{2n}-1} , & \text{if} \, \, q \, \equiv \,7 \, mod \,\, 16;\\
     	
     	C^{16}_{p^{n}-1}\times C^{8}_{p^{2n}-1} , & \text{if} \, \, q \, \equiv \,-7 \, mod \,\, 16.\\

     	\end{cases}  \right.   \]

     \end{theorem}

     \begin{proof}
     	The presentation of  $G\cong C_{16} \times C_{2}$ is given by $$C_{16} \times C_{2} = < a, b\,\,  |\,\,  a^{16} = b^{2} = 1, \,ab = ba>.$$
     	\begin{enumerate}
     		\item If $p=2$, then $FG$ is non-semisimple  and  $|F| = q = 2^n$. It is well known that  $U(FG) \cong V(FG) \times F^{*}$ and $|V(FG)| = 2^{31n}$ as $dim_{F}J(FG) = 31$. Obviously exponent of $V(FG)$ is 16.  Suppose $V(FG) \cong C^{l_{1}}_{16} \times C^{l_{2}}_{8} \times C^{l_{3}}_{4} \times C^{l_{4}}_{2} $ such that $2^{31n} = 16^{l_{1}} \times 8^{l_{2}} \times  4^{l_{3}} \times 2^{l_{4}}$. Now we  will compute $l_1$, $l_{2}$, $l_{3}$ and $l_{4}$. Set \\
     		$ W_{1} =  \big  \{\gamma_{1} \in \omega(G) : \gamma_{1}^2 =0 \, \text{and  there exists} \, \beta \in \omega(G), \,\,\newline \text{such that}\, \gamma_{1}  = \beta^{8}  \big  \}$,
     		$ W_{2} =  \big  \{\gamma_{2} \in \omega(G) : \gamma_{2}^2 =0 \, \text{and  there exists}\newline \, \beta \in \omega(G), \,\, \text{such that}\, \gamma_{2}  = \beta^{4}  \big  \}$ and
     		$ W_{3} =  \big  \{\gamma_{3} \in \omega(G) : \gamma_{3}^2 =0 \, \text{and  there exists} \, \beta \in \omega(G), \,\, \text{such that}\, \gamma_{3}  = \beta^{2}  \big  \}$.
     		Now if $\gamma  = \sum_{j = 0}^{1} \sum_{i = 0}^{15} \alpha_{16j+i} a^{i}b^{j}  \in \omega(G)$, then $ \sum_{i = 0}^{15} \alpha_{2i+j}  = 0$, for $j = 0, 1 $. Also  $ \gamma^{2}  = \sum_{j = 0}^{7} \sum_{i = 0}^{3} \alpha^{2}_{8i+j}a^{2j} $,  $  \gamma^{4}  = \sum_{j = 0}^{3} \sum_{i = 0}^{7} \alpha^{4}_{4i+j} a^{4j}$ and  $\gamma^{8}  = \sum_{j = 0}^{1} \sum_{i = 0}^{15} \alpha^{8}_{2i+j} a^{8j}$.
     		Let  $\beta  = \sum_{j = 0}^{1} \sum_{i = 0}^{15} \beta_{16j+i} a^{i}b^{j}$, such that $\gamma_{1} = \beta^8$. Now applying condition $\gamma_{1}^2 = 0$ and by direct computation we have $\alpha_{i} = 0 $, for all $i \neq 0, 8$ and $\alpha_{0} = \alpha_{8} $. Thus 
     		$  W_{1} = \big  \{  \alpha_{0} (1+a^8), \alpha_{0} \in F \big \}$, $|W_{1}| = 2^{n}$ and $l_{1} =  n$. Similarly, applying the conditions  $\gamma_{2} = \beta^4$ and $\gamma_{2}^2 = 0$, we have $\alpha_{i} = 0 $, for all $i \neq 0, 4$ and $\alpha_{0} = \alpha_{4} $.  Thus  $ W_{2} = \big  \{ \alpha_{0} (1+a^{4}), \alpha_{0} \in F \big \}$, $|W_{2}| = 2^{n}$ and $l_{2} = n.$
     		Again,  applying the conditions  $\gamma_{3} = \beta^8$ and $\gamma_{3}^2 = 0$.  We have $\alpha_{i} = 0 $, for all $i \neq 0, 2, 8, 10  $ and $\alpha_{0} = \alpha_{8} $, $\alpha_{2} = \alpha_{10}$. Thus  $  W_{3} = \big  \{(\alpha_{0} + \alpha_{2}a^{2}) (1+a^{8}), \alpha_{0},  \alpha_{2} \in F \big \}$,  $l_{3} = 2n$ and   $l_{4} = 20n$.
     		Hence $V(FG)  \cong C^{n}_{16}\times C^{n}_{8} \times C^{2n}_{4} \times C^{20n}_{2}$ and   hence the result. 
     		
     		\item If $p\neq 2$, then $|F|= p^{n}$. Using the similar arguments as in Theorem \ref{th1}, $F[C_{16} \times C_{2}]$ is semisimple and we have    $m$=16, $\sum_{i=1}^{r} [D_{i} : F] = 32$. By   observation we have following possibilities for $q$:
     	\end{enumerate}
     	(a)	  If $q \equiv   1 \,\,  mod \, \, 16$,       then $t = 1$;	\\
     	(b)	  If $q \equiv   -1 \,\,  mod \, \, 16$,       then $t = 2$;	\\
     	(c) If $q \equiv   3, \, -5\,  mod \, \, 16$,      then $t =4$;\\
     	(d) If $q \equiv  -3, \, 5 mod \, \, 16$,      then $t =4$;\\
     	(e) If $q \equiv  7 \,\,  mod \, \, 16$,      then $t =2$;\\
     	(f) If $q \equiv  -7 \,\,  mod \, \, 16$,      then $t =2$.\\

     	Hence we have the following cases:
     	
     	\begin{enumerate}
     		\item If $q\equiv \,1\, mod\,16$,  then $T = \{1\}\,mod\,\,16.$  Thus, $p$-regular $F$-conjugacy classes are the conjugacy classes of $C_{16} \times C_{2} $  and  $c$=32.   Hence  $F[C_{16} \times C_{2}] \cong  F^{32}.$  
     		
     		\item If $q\equiv \,-1\,mod\,16$,   then $T = \{1,-1\}\,mod\,\,16.$  Thus, $p$-regular $F$-conjugacy classes are 
     		$\{1\}$, $\{b\}$,  $\{a^{8}\}$, $\{a^{\pm i} \}$, where $1\leq i \leq 7$,  $\{a^{8}b \}$, $\{a^{j}b, a^{-j}b\}$,   where $1\leq j \leq 7$  and  $c$=18.     Hence  $F[C_{16} \times C_{2}] \cong F^{4}\oplus F^{14}_{2}. $
     		\item If $q \equiv   3, \, -5\,  mod \, \, 16$,   then $T = \{1,  3,  9, 11\}\,mod\,\,16.$  Thus,  $p$-regular $F$-conjugacy classes are 
     		$\{1\}$,  $\{b\}$,  $\{a, a^{3}, a^{-7}, a^{-5}\}$, \\$\{a^{-1}, a^{-3}, a^{5}, a^{7}\}$, $\{ a^{2}, a^{6}\}$,   $\{ a^{-2}, a^{-6}\}$,  $\{ a^{\pm4}\}$, $\{ a^{8}\}$, $\{ab, a^{3}b,\newline a^{-7}b, a^{-5}b\}$, $\{a^{-1}b, a^{-3}b, a^{5}b, a^{7}b\}$, $\{ a^{ 2}b, a^{6}b \}$, $\{ a^{-2}b, a^{-6}b\}$, $\{ a^{\pm4}b\}$, $\{ a^{8}b\}$ and   $c$=14 .  Hence  $F[C_{16} \times C_{2}] \cong   F^{6}_{2} \oplus F^{4}_{4}\oplus F^{4}. $
     		\item If $q \equiv   -3, \,  5\,  mod \, \, 16$,   then $T = \{1,  5,  9, 13\}\,mod\,\,16.$  Thus,  $p$-regular $F$-conjugacy classes are 
     		$\{1\}$,  $\{b\}$,  $\{a, a^{5}, a^{-3}, a^{-7}\}$, \\$\{a^{-1}, a^{-5}, a^{3}, a^{7}\}$, $\{ a^{2}, a^{-6}\}$,   $\{ a^{-2}, a^{6}\}$,  $\{ a^{ 4}\}$, $\{ a^{-4}\}$, $\{ a^{8}\}$,\\ $\{ab, a^{5}b, a^{-3}b, a^{-7}b\}$, $\{a^{-1}b, a^{-5}b, a^{3}b, a^{7}b\}$, $\{ a^{2}b, a^{-6}b \}$, $\{ a^{-2}b, a^{6}b\}$, $\{ a^{4}b\}$, $\{ a^{-4}b\}$, $\{ a^{8}b\}$ and   $c$=16 .   Hence  $F[C_{16} \times C_{2}] \cong   F^{4}_{2} \oplus F^{4}_{4}\oplus F^{8}. $

     		\item If $q \equiv   7 \,  mod \, \, 16$,   then $T = \{1,  7\}\,mod\,\,16.$  Thus,  $p$-regular $F$-conjugacy classes are 
     		$\{1\}$,  $\{b\}$,  $\{a, a^{7}\}$, $\{a^{3}, a^{5}\}$, $\{a^{-1}, a^{-7}\}$, $\{a^{-3}, a^{-5}\}$, $\{ a^{\pm2}\}$,   $\{a^{\pm6}\}$,  $\{ a^{\pm4}\}$, $\{ a^{8}\}$,  $\{ab, a^{7}b\}$, $\{a^{3}b, a^{5}b\}$, $\{a^{-1}b, a^{-7}b\}$, $\{a^{-3}b, a^{-5}b\}$, $\{ a^{\pm2}b\}$,   $\{a^{\pm6}b\}$,  $\{ a^{\pm4}b\}$, $\{ a^{8}b\}$ and   $c$=18 .   Hence  $F[C_{16} \times C_{2}] \cong   F^{14}_{2} \oplus F^{4}. $
     		
     		\item If $q \equiv   -7 \,  mod \, \, 16$,   then $T = \{1,  9\}\,mod\,\,16.$  Thus,  $p$-regular $F$-conjugacy classes are 
     		$\{1\}$,  $\{b\}$,  $\{a, a^{-7}\}$, $\{a^{3}, a^{-5}\}$, $\{a^{-1}, a^{7}\}$, $\{a^{-3}, a^{ 5}\}$, $\{ a^{2}\}$,   $\{ a^{-2}\}$,  $\{a^{6}\}$, $\{a^{-6}\}$,  $\{ a^{4}\}$,  $\{ a^{-4}\}$, $\{ a^{8}\}$,  $\{ab, a^{-7}b\}$, $\{a^{3}b, a^{-5}b\}$, $\{a^{-1}b, a^{ 7}b\}$, $\{a^{ -3}b, a^{ 5}b\}$, $\{ a^{2}b\}$,  $\{ a^{-2}b\}$,  $\{a^{6}b\}$, $\{a^{-6}b\}$,  $\{ a^{ 4}b\}$,  $\{ a^{-4}b\}$, $\{ a^{8}b\}$ and   $c$=24 .   Hence  $F[C_{16} \times C_{2}] \cong   F^{8}_{2} \oplus F^{16}. $
     		Thus we have the result.
     	\end{enumerate}
     \end{proof}
     
     \begin{theorem}
     	Let $F$ be a finite field of characteristic $p>0$ having  $q= p^{n}$ elements and $G \cong C_{8} \times C_{4}.$
     	\begin{enumerate}
     		\item If $p=2$. Then,
     		$$ U(F[C_{8} \times C_{4}])  \cong  C^{n}_{8} \times C^{5n}_{4} \times C^{18n}_{2}  \times C_{2^n -1}. $$
     		
     		\item If $p\neq 2$. Then, 
     		
     	\end{enumerate}

     	\[
     	U(F[C_{8} \times C_{4}]) \cong \left.
     	\begin{cases}
     	
     	C^{32}_{p^n-1},     & \text{if} \, \, q \, \equiv \, 1\, mod \,\, 8;\\
     	
     	C^{4}_{p^n-1}\times C^{14}_{p^{2n}-1},    & \text{if} \, \, q \, \equiv \, -1 \, mod \,\, 8;\\
     	
     	C^{4}_{p^n-1}\times C^{14}_{p^{2n}-1},    & \text{if} \, \, q \, \equiv \,3   \, mod \,\, 8;\\
     	
     	C^{16}_{p^n-1}\times C^{8}_{p^{2n}-1},  & \text{if} \, \, q \, \equiv \,-3 \, mod \,\, 8.\\
     	
     	\end{cases}  \right.   \]

     \end{theorem}

     \begin{proof}
     	The presentation of  $G\cong C_{8} \times C_{4}$ is given by $$C_{8} \times C_{4} = < a, b\,\,  |\,\,  a^{8} = b^{4} = 1, \,ab = ba>.$$
     	\begin{enumerate}
     		\item If $p=2$, then $FG$ is non-semisimple  and  $|F| = q = 2^n$. It is well known that  $U(FG) \cong V(FG) \times F^{*}$ and $|V(FG)| = 2^{31n}$ as $dim_{F}J(FG) = 31$. Obviously exponent of $V(FG)$ is 8.  Suppose $V(FG) \cong   C^{l_{1}}_{8} \times C^{l_{2}}_{4} \times C^{l_{3}}_{2} $ such that  $2^{31n} =   8^{l_{1}} \times  4^{l_{2}} \times 2^{l_{3}}$. Now we  will compute $l_1$, $l_{2}$  and $l_{3}$. Set  
     		$ W_{1} =  \big  \{\alpha  \in \omega(G) : \alpha^{2} =0 \, \text{and  there exists} \, \beta \in \omega(G), \,\, \text{such that}\, \alpha  = \beta^{4}  \big  \},$ 
     		$W_{2} =  \big  \{\gamma  \in \omega(G) : \gamma^{2} =0 \, \text{and  there exists} \, \beta \in \omega(G), \,\, \text{such that}\, \gamma   = \beta^{2}  \big  \}.$ \\
     		If $\alpha  = \sum_{j = 0}^{3} \sum_{i = 0}^{7} \alpha_{8j+i} a^{i}b^{j}  \in \omega(G)$, then $ \sum_{i = 0}^{7} \alpha_{4i+j}  = 0$, for $j = 0, 1, 2, 3 $. 
     		Let  $\beta  = \sum_{j = 0}^{3} \sum_{i = 0}^{7} \beta_{8j+i} a^{i}b^{j}$  such that $\alpha = \beta^4$. Now applying condition $\alpha^{2} = 0$, $\alpha  = \beta^{4}$ and  by direct computation, we have $\alpha_{i} = 0$, for all $i \neq 0, 4$  and $\alpha_{0} = \alpha_{4}$. Thus 
     		$ W_{1} = \big  \{  \alpha_{0} (1+a^4), \alpha_{0} \in F \big \}. $ 
     		Therefore $|W_{1}| = 2^{n}$ and $l_{1} =  n$. 
     		Similarly, applying the conditions $\gamma = \beta^{2}$ , $\gamma^{2} = 0$ and by direct computation,  we have  $|W_{2}| = 2^{5n}$,    $l_{2} = 5n$  and $l_{3} = 18n$.
     		Hence $V(FG)  \cong   C^{n}_{8} \times C^{5n}_{4} \times C^{18n}_{2}$ and hence the result. 
     		
     		\item If $p\neq 2$, then $|F|= p^{n}$. Using the similar arguments as in Theorem \ref{th1}, $F[C_{8} \times C_{4}]$ is semisimple and we have   $m$=8, $\sum_{i=1}^{r} [D_{i} : F] = 32$. By   observation we have following possibilities for $q$:
     	\end{enumerate}	
     	(a)	  If $q \equiv   1 \,\,  mod \, \, 8$,       then $t = 1$;	\\
     	(b)	  If $q \equiv   -1 \,\,  mod \, \, 8$,       then $t = 2$;	\\
     	(c) If $q \equiv   3  \,  mod \, \, 8$,      then $t =2$;\\
     	(d) If $q \equiv  -3  \,   mod \, \, 8$,      then $t =2$.\\

     	Hence we have the following cases:
     	
     	\begin{enumerate}
     		\item If $q\equiv \,1\, mod\,8$,  then $T = \{1\}\,mod\,\,8.$  Thus, $p$-regular $F$-conjugacy classes are the conjugacy classes of $C_{8} \times C_{4} $  and  $c$=32.   Hence  $F[C_{8} \times C_{4}] \cong  F^{32}.$  
     		
     		\item If $q\equiv \,-1\,mod\,8$,   then $T = \{1,-1\}\,mod\,\,8.$  Thus, $p$-regular $F$-conjugacy classes are 	$\{1\}$,  $\{b^{2}\}$, $\{b, b^{3}\}$,  $\{a^{\pm1} \}$,  $\{a^{\pm2} \}$,  $\{a^{\pm3} \}$,   $\{a^{4} \}$, $\{ab, a^{-1}b^{3}\}$,   $\{a^{2}b, a^{-2}b^{3}\}$, $\{a^{3}b, a^{-3}b^{3}\}$, $\{a^{4}b, a^{4}b^{3}\}$, $\{a^{-3}b, a^{3}b^{3}\}$, $\{a^{-2}b, a^{2}b^{3}\}$, $\{a^{-1}b,  ab^{3}\}$, $\{ab^{2}, a^{-1}b^{2}\}$, $\{a^{-2}b^{2}, a^{2}b^{2}\}$, $\{a^{3}b^{2}, a^{-3}b^{2}\}$, $\{a^{4}b^{2}\}$  and  $c$=18.    Hence  $F[C_{8} \times C_{4}] \cong F^{4}\oplus F^{14}_{2}.$
     		\item If $q\equiv \,3\,mod\,8$,   then $T = \{1,3\}\,mod\,\,8.$  Thus, $p$-regular $F$-conjugacy classes are $\{1\}$,  $\{b^{2}\}$, $\{b, b^{3}\}$,  $\{a, a^{3} \}$,  $\{a^{2}, a^{-2} \}$,  $\{a^{-1}, a^{-3} \}$,   $\{a^{4} \}$,  $\{ab, a^{3}b^{3} \}$,  $\{a^{2}b, a^{-2}b^{3} \}$,  $\{a^{-1}b, a^{-3}b^{3}\}$,   $\{a^{4}b, a^{4}b^{3} \}$,  $\{ab^{3}, a^{3}b \}$,  $\{a^{2}b^{3}, a^{-2}b \}$,  $\{a^{-1}b^{3}, a^{-3}b \}$, $\{ab^{2}, a^{3}b^{2} \}$,  $\{a^{2}b^{2}, a^{-2}b^{2} \}$,  $\{a^{-1}b^{2}, a^{-3}b^{2}\}$,   $\{ a^{4}b^{2} \}$  and  $c$=18.    Hence  $F[C_{8} \times C_{4}] \cong F^{4}\oplus F^{14}_{2}. $
     		\item If $q\equiv \,-3\,mod\,8$,   then $T = \{1,5\}\,mod\,\,8.$  Thus, $p$-regular $F$-conjugacy classes are $\{1\}$,  $\{b\}$, $\{b^{2}\}$, $\{b^{3}\}$,  $\{a, a^{-3} \}$,  $\{a^{2}\}$, $\{a^{-2} \}$,   $\{a^{-1}, a^{3}\}$,   $\{a^{4} \}$,  $\{ab, a^{-3}b \}$,  $\{a^{2}b\}$, $\{a^{-2}b \}$,   $\{a^{-1}b, a^{3}b\}$,   $\{a^{4}b \}$,  $\{ab^{2}, a^{-3}b^{2} \}$,  $\{a^{2}b^{2}\}$, $\{a^{-2}b^{2} \}$,   $\{a^{-1}b^{2}, a^{3}b^{2}\}$,   $\{a^{4}b^{2} \}$,  $\{ab^{3}, a^{-3}b^{3} \}$,  $\{a^{2}b^{3}\}$, $\{a^{-2}b^{3} \}$,   $\{a^{-1}b^{3}, a^{3}b^{3}\}$,   $\{a^{4}b^{3} \}$ and  $c$=24.    Hence  $F[C_{8} \times C_{4}] \cong F^{16}\oplus F^{8}_{2}. $\\
     		Thus we have the result.
     	\end{enumerate}
     \end{proof}

     \begin{theorem}
     	Let $F$ be a finite field of characteristic $p>0$ having  $q= p^{n}$ elements and $G \cong C_{8} \times C_{2} \times C_{2}.$
     	\begin{enumerate}
     		\item If $p=2$. Then,
     		$$ U(F[C_{8} \times C_{2} \times C_{2}])  \cong  C^{n}_{8} \times C^{ n}_{4} \times C^{26n}_{2}  \times C_{2^n -1}. $$
     		
     		\item If $p\neq 2$. Then, 
     		
     	\end{enumerate}
     	
     	\[
     	U(F[C_{8} \times C_{2} \times C_{2}]) \cong \left.
     	\begin{cases}
     	
     	C^{32}_{p^n-1},    & \text{if} \, \, q \, \equiv \, 1\, mod \,\, 8;\\
     	
     	C^{8}_{p^n-1}\times C^{12}_{p^{2n}-1},   & \text{if} \, \, q \, \equiv \, -1 \, mod \,\, 8;\\
     	
     	C^{8}_{p^n-1}\times C^{12}_{p^{2n}-1},   & \text{if} \, \, q \, \equiv \,3   \, mod \,\, 8;\\
     	
     	C^{16}_{p^n-1}\times C^{8}_{p^{2n}-1},  & \text{if} \, \, q \, \equiv \,-3 \, mod \,\, 8.\\

     	\end{cases}  \right.   \]

     \end{theorem}

     \begin{proof}
     	The presentation of  $G\cong C_{8} \times C_{2} \times C_{2}$ is given by $$C_{8} \times C_{2} \times C_{2} = < a, b, c\,\,  |\,\,  a^{8} = b^{2} = c^{2} = 1, \,ab = ba, bc = cb, ac = ca>.$$
     	\begin{enumerate}
     		\item If $p=2$, then $FG$ is non-semisimple   and  $|F| = q = 2^n$. It is well known that  $U(FG) \cong V(FG) \times F^{*}$ and $|V(FG)| = 2^{31n}$ as $dim_{F}J(FG) = 31$. Obviously exponent of $V(FG)$ is 8.  Suppose $V(FG) \cong   C^{l_{1}}_{8} \times C^{l_{2}}_{4} \times C^{l_{3}}_{2} $ such that  $2^{31n} =   8^{l_{1}} \times  4^{l_{2}} \times 2^{l_{3}}$. Now we  will compute $l_1$, $l_{2}$  and $l_{3}$. Set  
     		$ W_{1} =  \big  \{\alpha  \in \omega(G) : \alpha^{2} =0 \, \text{and  there exists} \, \beta \in \omega(G), \,\, \text{such that}\, \alpha  = \beta^{4}  \big  \},$  
     		$W_{2} =  \big  \{\gamma  \in \omega(G) : \gamma^{2} =0 \, \text{and  there exists} \, \beta \in \omega(G), \,\, \text{such that}\, \gamma   = \beta^{2}  \big  \}.$ \\
     		Let $\alpha  =  \sum_{k = 0}^{1} \sum_{j = 0}^{1} \sum_{i = 0}^{7} \alpha_{8 (j+2k)+i} a^{i}b^{j}c^{k}  \in \omega(G)$ and \\ $\beta  = \sum_{k = 0}^{1} \sum_{j = 0}^{1} \sum_{i = 0}^{7} \beta_{8 (j+2k)+i} a^{i}b^{j}c^{k}$  such that $\alpha = \beta^4$. Now applying the  conditions $\alpha^{2} = 0$, $\alpha  = \beta^{4}$ and  by direct computation, we have $\alpha_{i} = 0$, for all $i \neq 0, 4$  and $\alpha_{0} = \alpha_{4}$. Thus 
     		$ W_{1} = \big  \{  \alpha_{0} (1+a^4), \alpha_{0} \in F \big \}. $ 
     		Therefore $|W_{1}| = 2^{n}$ and $l_{1} =  n$. 
     		Similarly, applying the conditions $\gamma = \beta^{2}$ , $\gamma^{2} = 0$ and by direct computation,  we have $\alpha_{i} = 0$, for all $i \neq 0, 2$ and  $\alpha_{0} = \alpha_{2}$.  Thus 
     		$  W_{2} = \big  \{ \alpha_{0}(1+a^2), \alpha_{0} \in F \big \}.$  
     		Therefore $|W_{2}| = 2^{n}$,    $l_{2} =  n$  and $l_{3} = 26n$.
     		Hence $V(FG)  \cong   C^{n}_{8} \times C^{n}_{4} \times C^{26n}_{2}$ and hence  the result follows. 
     		
     		\item If $p\neq 2$, then $|F|= p^{n}$. Using the similar arguments as in Theorem \ref{th1}, $F[C_{8} \times C_{2} \times C_{2}]$ is semisimple   and $m$=8, $\sum_{i=1}^{r} [D_{i} : F] = 32$. Here the number of  $p$-regular $F$-conjugacy classes,  denoted  by $w$. By   observation we have following possibilities for $q$:
     	\end{enumerate}
     	(a)	  If $q \equiv   1 \,\,  mod \, \, 8$,       then $t = 1$;	\\
    	(b)	  If $q \equiv   -1 \,\,  mod \, \, 8$,       then $t = 2$;	\\
     	(c) If $q \equiv   3  \,  mod \, \, 8$,      then $t =2$;\\
     	(d) If $q \equiv  -3  \,   mod \, \, 8$,      then $t =2$.\\
     	Now we have the cases:
     	\begin{enumerate}
     		\item If $q\equiv \,1\, mod\,8$,  then $T = \{1\}\,mod\,\,8.$  Thus, $p$-regular $F$-conjugacy classes are the conjugacy classes of $C_{8} \times C_{2} \times C_{2} $  and  $w$=32.   Hence  $F[C_{8} \times C_{2} \times C_{2}] \cong  F^{32}.$  
     		
     		\item If $q\equiv \,-1\,mod\,8$,   then $T = \{1, 7\}\,mod\,\,8.$  Thus, $p$-regular $F$-conjugacy classes are 	$\{1\}$,  $\{a, a^{7} \}$, $\{a^{2}, a^{6} \}$, $\{a^{3}, a^{5} \}$, $\{a^{4}\}$, $\{b\}$, $\{c\}$, $\{ab, a^{7}b \}$, $\{a^{2}b, a^{6}b \}$, $\{a^{3}b, a^{5}b \}$, $\{a^{4}b \}$, $\{ac, a^{7}c \}$, $\{a^{2}c, a^{6}c \}$, $\{a^{3}c, a^{5}c\}$, $\{a^{4}c\}$, $\{bc \}$, $\{abc, a^{7}bc \}$, $\{a^{2}bc, a^{6}bc \}$, $\{a^{3}bc, a^{5}bc\}$, $\{a^{4}bc\}$  and  $w$=20.    Hence  $F[C_{8} \times C_{2} \times C_{2}] \cong F^{8}\oplus F^{12}_{2}.$
     		\item If $q\equiv \,3\,mod\,8$,   then $T = \{1,3\}\,mod\,\,8.$  Thus, $p$-regular $F$-conjugacy classes are	$\{1\}$,  $\{a, a^{3} \}$, $\{a^{2}, a^{6} \}$, $\{a^{5}, a^{7} \}$, $\{a^{4}\}$, $\{b\}$, $\{c\}$, $\{ab, a^{3}b \}$, $\{a^{2}b, a^{6}b \}$, $\{a^{5}b, a^{7}b \}$, $\{a^{4}b \}$, $\{ac, a^{3}c \}$, $\{a^{2}c, a^{6}c \}$, $\{a^{5}c, a^{7}c\}$, $\{a^{4}c\}$, $\{bc \}$, $\{abc, a^{3}bc \}$, $\{a^{2}bc, a^{6}bc \}$, $\{a^{5}bc, a^{7}bc\}$, $\{a^{4}bc\}$ and  $w$=20.   Hence  $F[C_{8} \times C_{2} \times C_{2}] \cong F^{8}\oplus F^{12}_{2}. $
     		\item If $q\equiv \,-3\,mod\,8$,   then $T = \{1,5\}\,mod\,\,8.$  Thus, $p$-regular $F$-conjugacy classes are 	$\{1\}$,  $\{a, a^{5} \}$, $\{a^{2} \}$, $\{a^{6} \}$, $\{a^{3}, a^{7} \}$, $\{a^{4}\}$, $\{b\}$, $\{c\}$, $\{ab, a^{5}b \}$, $\{a^{2}b\}$, $\{ a^{6}b \}$, $\{a^{3}b, a^{7}b \}$, $\{a^{4}b \}$, $\{ac, a^{5}c \}$, $\{a^{2}c\}$, $\{a^{6}c \}$, $\{a^{3}c, a^{7}c\}$, $\{a^{4}c\}$, $\{bc\}$, $\{abc, a^{5}bc \}$, $\{a^{2}bc\}$, $\{a^{6}bc \}$, $\{a^{3}bc, a^{7}bc\}$, $\{a^{4}bc\}$ and  $w$=24.    Hence  $F[C_{8} \times C_{2} \times C_{2}] \cong F^{16}\oplus F^{8}_{2}. $\\
     		Thus we have the result.
     	\end{enumerate}
     \end{proof}

     \begin{theorem}
     	Let $F$ be a finite field of characteristic $p>0$ having  $q= p^{n}$ elements and $G \cong C_{4}^{2}  \times C_{2}.$
     	\begin{enumerate}
     		\item If $p=2$. Then,
     		$$ U(F[C_{4}^{2}\times C_{2}])  \cong  C^{3n}_{4}   \times C^{25n}_{2}  \times C_{2^n -1}. $$
     		
     		\item If $p\neq 2$. Then, 
     		
     	\end{enumerate}
    	\[
     	U(F[C_{4}^{2} \times C_{2}]) \cong \left.
     	\begin{cases}
     	
     	C^{32}_{p^n-1},    & \text{if} \, \, q \, \equiv \, 1\, mod \,\, 4;\\
     	
     	C^{8}_{p^n-1}\times C^{12}_{p^{2n}-1},   & \text{if} \, \, q \, \equiv \, -1 \, mod \,\, 4.\\
     
     	\end{cases}  \right.   \]
     	
     \end{theorem}
     
     \begin{proof}
     	The presentation of  $G\cong C_{4} \times C_{4} \times C_{2}$ is given by $$C_{4}^{2} \times C_{2} = < a, b, c\,\,  |\,\,  a^{4} = b^{4} = c^{2} = 1, \,ab = ba, bc = cb, ac = ca>.$$
     	\begin{enumerate}
     		\item If $p=2$, then $FG$ is non-semisimple  and  $|F| = q = 2^n$. It is well known that  $U(FG) \cong V(FG) \times F^{*}$ and $|V(FG)| = 2^{31n}$ as $dim_{F}J(FG) = 31$. Obviously exponent of $V(FG)$ is 4.  Suppose $V(FG) \cong   C^{l_{1}}_{4} \times C^{l_{2}}_{2}$ such that  $2^{31n} =   4^{l_{1}} \times  2^{l_{2}}$. Now we  will compute $l_1$  and $l_{2}$. Set  
     		$  W  =  \big  \{\alpha  \in \omega(G) : \alpha^{2} =0 \, \text{and  there exists} \, \beta \in \omega(G), \,\, \text{such that}\, \alpha  = \beta^{2}  \big  \}.$  
 If  $\alpha  =  \sum_{k = 0}^{1} \sum_{j = 0}^{3} \sum_{i = 0}^{3} \alpha_{4(j+4k)+i} a^{i}b^{j}c^{k}  \in \omega(G)$, then $\sum_{i = 0}^{3} \alpha_{2(j+2k)+i} = 0$, for $j = 0, 1, 2, 3$  and  $k = 0, 1$. Let $\beta  = \sum_{k = 0}^{1} \sum_{j = 0}^{3} \sum_{i = 0}^{3} \beta_{4(j+4k)+i} a^{i}b^{j}c^{k}$  such that $\alpha = \beta^2$. Now applying the  conditions $\alpha^{2} = 0$, $\alpha  = \beta^{2}$ and  by direct computation, we have $\alpha_{i} = 0$, for all $i \neq 0, 2, 8, 10 $  and $\alpha_{0} = \alpha_{2}$.\ Thus
     		$ W  = \big  \{  \alpha_{0} (1+a^2) + (\alpha_{8} + \alpha_{10}a^{2})b^{2}, \alpha_{0}, \alpha_{8}, \alpha_{10} \in F \big \}. $ 
     		Therefore $|W| = 2^{3n}$, $l_{1} =  3n$ and  $l_{2} = 25n$.  
     		Hence $V(FG)  \cong   C^{3n}_{4}  \times C^{25n}_{2}$ and  the result follows. 
     		
     		\item If $p\neq 2$, then $|F|= p^{n}$. Using the similar arguments as in Theorem \ref{th1}, $F[C_{4} \times C_{4} \times C_{2}]$ is semisimple and   $m$=4, $\sum_{i=1}^{r} [D_{i} : F] = 32$ . By   observation we have following possibilities for $q$:
     	\end{enumerate}
     	(a)	  If $q \equiv   1 \,\,  mod \, \, 4$,       then $t = 1$;	\\
     	(b)	  If $q \equiv   -1 \,\,  mod \, \, 4$,       then $t = 2$.	\\
     
     	Now we have the cases:
     	\begin{enumerate}
     		\item If $q\equiv \,1\, mod\,4$,  then $T = \{1\}\,mod\,\,4.$  Thus, $p$-regular $F$-conjugacy classes are the conjugacy classes of $C_{4} \times C_{4} \times C_{2} $  and  $w$=32.   Hence  $F[C_{4} \times C_{4} \times C_{2}] \cong  F^{32}.$  
     		
     		\item If $q\equiv \,-1\,mod\,4$,   then $T = \{1, 3\}\,mod\,\,4.$  Thus, $p$-regular $F$-conjugacy classes are 	$\{1\}$,  $\{a, a^{3}\}$, $\{a^{2}  \}$, $\{b, b^{3} \}$, $\{ b^{2} \}$, $\{c\}$, $\{ab, a^{3}b^{3}\}$, $\{ab^{2}, a^{3}b^{2}\}$, $\{ab^{3}, a^{3}b \}$, $\{a^{2}b, a^{2}b^{3} \}$,  $\{ a^{2}b^{2} \}$, $\{bc, b^{3}c\}$, $\{ b^{2}c \}$, $\{abc, a^{3}b^{3}c \}$, $\{ab^{2}c, a^{3}b^{2}c \}$, $\{ab^{3}c, a^{3}bc \}$, $\{a^{2}bc, a^{2}b^{3}c \}$, $\{ a^{2}b^{2}c \}$, $\{ac, a^{3}c \}$, $\{  a^{2}c \}$  and  $w$=20.    Hence  $F[C_{4} \times C_{4} \times C_{2}] \cong F^{8}\oplus F^{12}_{2}.$
     		
     		Thus we have the result.
     	\end{enumerate}
     \end{proof}
 
     \begin{theorem}
     	Let $F$ be a finite field of characteristic $p>0$ having  $q= p^{n}$ elements and $G \cong C_{4} \times C^{3}_{2}.$
     	\begin{enumerate}
     		\item If $p=2$. Then,
     		$$ U(F[C_{4} \times C^{3}_{2}])  \cong   C^{n}_{4}  \times C^{29n}_{2}  \times C_{2^n -1}. $$
     		
     		\item If $p\neq 2$. Then, 
     		
     	\end{enumerate}
     	
     	\[
     	U(F[C_{4} \times  C^{3}_{2}]) \cong \left.
     	\begin{cases}
     	
     	C^{32}_{p^n-1},    & \text{if} \, \, q \, \equiv \, 1\, mod \,\, 4;\\
     	
     	C^{16}_{p^n-1}\times C^{8}_{p^{2n}-1},   & \text{if} \, \, q \, \equiv \, -1 \, mod \,\, 4.\\

     	\end{cases}  \right.   \]

     \end{theorem}

     \begin{proof}
     	The presentation of  $G\cong C_{4} \times  C^{3}_{2}$ is given by $$C_{4} \times C^{3}_{2}   = < a, b, c, d \,\,  |\,\,  a^{4} = b^{2} = c^{2}= d^{2} = 1, \,ab = ba, bc = cb, dc = cd,  ad = da>.$$
     	\begin{enumerate}
     		\item If $p=2$, then $FG$ is non-semisimple   and  $|F| = q = 2^n$. It is well known that  $U(FG) \cong V(FG) \times F^{*}$ and $|V(FG)| = 2^{31n}$ as $dim_{F}J(FG) = 31$. Obviously exponent of $V(FG)$ is 4.  Suppose $V(FG) \cong   C^{l_{1}}_{4} \times C^{l_{2}}_{2}$ such that  $2^{31n} =   4^{l_{1}} \times  2^{l_{2}}$. Now we  will compute $l_1$ and $l_{2}$. Set  
     		$$  W  =  \big  \{\alpha  \in \omega(G) : \alpha^{2} =0 \, \text{and  there exists} \, \beta \in \omega(G), \,\, \text{such that}\, \alpha  = \beta^{2}  \big  \}.$$ 
     		Let $\alpha  = \sum_{s = 0}^{1} \sum_{k = 0}^{1} \sum_{j = 0}^{1} \sum_{i = 0}^{3} \alpha_{4(j+2(k+2s))+i} a^{i}b^{j}c^{k}d^{s}  \in \omega(G)$ and  $\beta  =  \sum_{s = 0}^{1}\sum_{k = 0}^{1} \sum_{j = 0}^{1} \sum_{i = 0}^{3} \beta_{4(j+2(k+2s))+i} a^{i}b^{j}c^{k}d^{s}$  such that $\alpha = \beta^2$. Now applying the  conditions $\alpha^{2} = 0$, $\alpha  = \beta^{2}$ and  by direct computation, we have $\alpha_{i} = 0$, for all $i \neq 0, 2 $  and $\alpha_{0} = \alpha_{2}$. Thus $ W  = \big  \{  \alpha_{0} (1+a^2),   \alpha_{0} \in F \big \}. $ 
     		Therefore $|W| = 2^{n}$, $l_{1} =   n$ and  $l_{2} = 29n$.  
     		Hence $V(FG)  \cong   C^{n}_{4}  \times C^{29n}_{2}$ and  the result follows. 
     		
     		\item If $p\neq 2$, then $|F|= p^{n}$. Using the similar arguments as in Theorem \ref{th1}, $F[C_{4} \times C^{3}_{3} ]$ is semisimple and  $m$=4, $\sum_{i=1}^{r} [D_{i} : F] = 32$. By   observation we have following possibilities for $q$:
     	\end{enumerate}
     	(a)	  If $q \equiv   1 \,\,  mod \, \, 4$,       then $t = 1$;	\\
     	(b)	  If $q \equiv   -1 \,\,  mod \, \, 4$,       then $t = 2$.	\\
     	
     Now have the following cases:
     	
     	\begin{enumerate}
     		\item If $q\equiv \,1\, mod\,4$,  then $T = \{1\}\,mod\,\,4.$  Thus, $p$-regular $F$-conjugacy classes are the conjugacy classes of $C_{4} \times C^{3}_{2} $  and  $w$=32.   Hence  $F[C_{4} \times C^{3}_{2}] \cong  F^{32}.$  
     		
     		\item If $q\equiv \,-1\,mod\,4$,   then $T = \{1, 3\}\,mod\,\,4.$  Thus, $p$-regular $F$-conjugacy classes are 	$\{1\}$,  $\{a, a^{3}\}$, $\{a^{2}  \}$ , $\{ b\}$, $\{c\}$, $\{d\}$,  $\{ab, a^{3}b \}$, $\{a^{2}b \}$, $\{ac, a^{3}c \}$, $\{a^{2}c \}$,   $\{ad, a^{3}d\}$, $\{ a^{2}d \}$, $\{ bc \}$, $\{cd \}$, $\{bd \}$, $\{ abc, a^{3}bc \}$, $\{ a^{2}bc\}$, $\{ acd, a^{3}cd \}$, $\{ a^{2}cd\}$, $\{ abd, a^{3}bd \}$, $\{ a^{2}bd \}$, $\{ bcd \}$, $\{ abcd, a^{3}bcd \}$, $\{ a^{2}bcd\}$ and  $w$=24.   Hence  $F[C_{4} \times C^{3}_{2}] \cong F^{16}\oplus F^{8}_{2}.$
     		
     		Hence we have the result.
     	\end{enumerate}
     \end{proof}

     \begin{theorem}
     	Let $F$ be a finite field of characteristic $p>0$ having  $q= p^{n}$ elements and $G \cong C^{5}_{2}.$
     	\begin{enumerate}
     		\item If $p=2$. Then,
     		$ U(F[C^{5}_{2}])  \cong   C^{31n}_{2}   \times C_{2^n -1} $.
     		
     		\item If $p\neq 2$. Then, 
     		
     	\end{enumerate}
     	
     	$U(F[C^{5}_{2}]) \cong C^{32}_{p^{n} -1} $, if $q \equiv \, 1\, mod\,  2.$
     	
     \end{theorem}

     \begin{proof}
     	The presentation of  $G\cong C^{5}_{2}$ is given by $C^{5}_{2} = < a, b, c, d, e \,\,  |\,\,  a^{2} = b^{2} = c^{2}= d^{2} = e^{2} = 1, \,ab = ba, bc = cb, dc = cd,  ed = de, ea = ae>.$
     	\begin{enumerate}
     		\item If $p=2$, then $FG$ will be non-semisimple in this case and  $|F| = q = 2^n$. Since $G\cong C^{5}_{2}$, therefore by Lemma \ref{l1}, we have 
     		$U(FG)  \cong   C^{31n}_{2}  \times C_{2^n -1}$.  
     		
     		\item If $p\neq 2$, then $|F|= p^{n}$. Using the similar arguments as in Theorem \ref{th1}, $F[C^{5}_{2}]$ is semisimple and $m$=2, $\sum_{i=1}^{r} [D_{i} : F] = 32$. By   observation we have $q \equiv   1 \,\,  mod \, \, 2$ and $t = 1$.
     	\end{enumerate}
     	Hence $q\equiv \,1\, mod\,2$,  implies  $T = \{1\}\,mod\,\,2.$  Thus, $p$-regular $F$-conjugacy classes are the conjugacy classes of $C^{5}_{2}$  and  $w$=32.  Therefore,  $F[C^{5}_{2}] \cong  F^{32}$ and we  have the result. 
     	
     \end{proof}
    
  \bibliographystyle{acm}
  \bibliography{SB_32_A}

\begin{thebibliography}{10}

\bibitem{saha2020}
{\sc Ansari, S.~F., and Sahai, M.}
\newblock Unit groups of group algebras of groups of order 20.
\newblock {\em Quaes. Math.\/} (2020), 1--9.

\bibitem{Bhatt}
{\sc Bhatt, S., and Chandra, H.}
\newblock Structure of unit group of {FpnD60}.
\newblock {\em Asian-European Journal of Mathematics
  https://doi.org/10.1142/S1793557121500753\/} (2020), 2150075.

\bibitem{Creedon2008}
{\sc Creedon, L.}
\newblock The unit group of small group algebras and the minimum counter
  example to the isomorphism problem.
\newblock {\em International Journal of pure and Applied Mathematics vol. 49\/}
  (2008), 531--537.

\bibitem{rafarz1}
{\sc Ferraz, R.~A.}
\newblock Simple components of the center of {FG/J(FG)}.
\newblock {\em Commun. Algebra vol. 36}, no. 9 (2008), 3191--3199.

\bibitem{gkarpi1}
{\sc Karpilovsky, G.}
\newblock Unit groups of classical rings.
\newblock {\em Clarendon Press, Oxford\/} (1988).

\bibitem{Gkarp2}
{\sc Karpilovsky, G.}
\newblock Group representations. introduction to group representations and
  characters.
\newblock {\em North Holland Mathematics Studies. Amsterdam, North Holland
  Publishing Co. vol. 1\/} (1992).

\bibitem{Mkhan1}
{\sc Khan, M.}
\newblock Structure of the unit group of {FD10}.
\newblock {\em Serdica Math. J. vol. 35}, no. 1 (2009), 15--24.

\bibitem{Khan2007}
{\sc Khan, M., Sharma, R.~K., and Srivastava, J.~B.}
\newblock The unit group of {FS4}.
\newblock {\em Acta Mathematica Hungarica vol. 118}, no. 1-2 (sep 2007),
  105--113.

\bibitem{NehaShajb1}
{\sc Makhijani, N., Sharma, R.~K., and Srivastava, J.~B.}
\newblock The unit group of algebra of circulant matrices.
\newblock {\em Int. J. Group Theory vol. 3}, no. 3 (2014), 25--34.

\bibitem{NehashaJb2}
{\sc Makhijani, N., Sharma, R.~K., and Srivastava, J.~B.}
\newblock The unit group of {Fq[D30]}.
\newblock {\em Serdica Math. J. vol. 41\/} (2015).

\bibitem{CPMSK}
{\sc Milies, C.~P., and Sehgal, S.~K.}
\newblock An introduction to group rings. algebras and applications.
\newblock {\em Bol. Soc. Brasil. Mat.\/} (2002).

\bibitem{Sahai2020}
{\sc Sahai, M., and Ansari, S.~F.}
\newblock Unit groups of finite group algebras of abelian groups of order at
  most 16.
\newblock {\em Asian-European Journal of Mathematics\/} (2020), 2150030.

\bibitem{Sharma2006}
{\sc Sharma, R.~K., Srivastava, J.~B., and Khan, M.}
\newblock The unit group of { FA4}.
\newblock {\em Publ. Math. Debrecen\/} (2006), 1–6.

\bibitem{Gtang1}
{\sc Tang, G., and Gao, Y.}
\newblock The unit group of {FG} of group with order 12.
\newblock {\em Int. J.Pure Appl. Math. vol. 73}, no. 1 (2011), 143--158.

\bibitem{Tang2014}
{\sc Tang, G., Wei, Y., and Li, Y.}
\newblock Unit groups of group algebras of some small groups.
\newblock {\em Czechoslovak Mathematical Journal vol. 64}, no. 1 (mar 2014),
  149--157.

\end{thebibliography}
    \end{document}